\documentclass[12pt]{amsart}
\usepackage{enumerate,amssymb,version,aliascnt,geometry}
\usepackage[all]{xy}
\usepackage{hyperref}
\hypersetup{
pdfauthor={Olivier Haution},
pdftitle={Involutions of varieties and Rost's degree formula},
pdfkeywords={Degree formula, Segre class, Steenrod squares},
hidelinks,
}

\DeclareMathOperator{\id}{id}
\DeclareMathOperator{\Spec}{Spec}
\DeclareMathOperator{\Proj}{Proj}
\DeclareMathOperator{\im}{im}
\DeclareMathOperator{\Sym}{Sym}
\DeclareMathOperator{\Zo}{\mathcal{Z}}
\DeclareMathOperator{\CH}{CH}
\DeclareMathOperator{\coker}{coker}
\DeclareMathOperator{\Sq}{Sq}

\newcommand{\inv}{^G}
\newcommand{\Oc}{\mathcal{O}}
\newcommand{\Zz}{\mathbb{Z}}
\newcommand{\Ic}{\mathcal{I}}
\newcommand{\Jc}{\mathcal{J}}
\newcommand{\Fc}{\mathcal{F}}
\newcommand{\Rg}[1]{\mathcal{R}_{#1}}
\newcommand{\Rgg}[1]{S_{#1}}
\newcommand{\Tan}{T}
\newcommand{\Lg}[1]{\mathcal{L}_{#1}}
\newcommand{\Hg}[1]{H_{#1}}
\newcommand{\St}[1]{\mathcal{S}_{#1}}
\newcommand{\Std}[1]{\mathcal{P}_{#1}}
\newcommand{\Au}{\mathbb{A}^1}
\newcommand{\et}[1]{\operatorname{Sq}({#1})}
\newcommand{\etd}[1]{s_{#1}}

\newtheorem{theorem}{Theorem}[section]
\newaliascnt{proposition}{theorem}
\newtheorem{proposition}[proposition]{Proposition}
\aliascntresetthe{proposition}

\newaliascnt{lemma}{theorem}
\newtheorem{lemma}[lemma]{Lemma}
\aliascntresetthe{lemma}

\newaliascnt{corollary}{theorem}
\newtheorem{corollary}[corollary]{Corollary}
\aliascntresetthe{corollary}

\theoremstyle{definition}

\newaliascnt{remark}{theorem}
\newtheorem{remark}[remark]{Remark}
\aliascntresetthe{remark}

\newaliascnt{example}{theorem}
\newtheorem{example}[example]{Example}
\aliascntresetthe{example}

\newaliascnt{definition}{theorem}
\newtheorem{definition}[definition]{Definition}
\aliascntresetthe{definition}

\newaliascnt{notation}{theorem}
\newtheorem{notation}[notation]{Notation}
\aliascntresetthe{notation}

\begin{document}
\begin{abstract}
To an algebraic variety equipped with an involution, we associate a cycle class in the modulo two Chow group of its fixed locus. This association is functorial with respect to proper morphisms having a degree and preserving the involutions. Specialising to the exchange involution of the square of a complete variety, we obtain Rost's degree formula in arbitrary characteristic (this formula was proved by Rost/Merkurjev in characteristic not two). 
\end{abstract}
\author{Olivier Haution}
\title{Involutions of varieties and Rost's degree formula}
\email{olivier.haution at gmail.com}
\address{Mathematisches Institut, Ludwig-Maximilians-Universit\"at M\"unchen, Theresienstr.\ 39, D-80333 M\"unchen, Germany}

\maketitle

\section*{Introduction}
The degree formula provides a connection between arithmetic and geometric properties of algebraic varieties, by relating the degrees of closed points, the values of characteristic numbers, and the degrees of rational maps. It has applications to questions of algebraic nature, in particular concerning quadratic forms.

When $X$ is a smooth, connected, complete variety, its Segre number $\etd{X}$ is the degree of the highest Chern class of the opposite of its tangent bundle:
\[
\etd{X} = \deg c_{\dim X}(-\Tan_X) \in \Zz.
\]
The index $n_X$ of $X$ is the g.c.d.\ of the degrees of the closed points of $X$; it is a divisor of $\etd{X}$. The degree formula provides a relation between $\etd{X}$ (a geometric, or topological, invariant) and $n_X$ (an arithmetic invariant). More precisely, let $f\colon Y \dasharrow X$ be a rational map of connected (smooth complete) varieties of the same dimension. Then $n_X \mid n_Y$. The degree of $f$, denoted $\deg f$, is defined as zero when $f$ is not dominant, and as the degree of the function field extension otherwise. The \emph{degree formula} is the relation in $\Zz/2$
\[
\frac{n_Y}{n_X} \cdot \frac{\etd{Y}}{n_Y} = \deg f \cdot \frac{\etd{X}}{n_X} \mod 2.
\]

A typical application of the degree formula is the following: if the integer $\etd{X}/n_X$ is odd, then the variety $X$ is incompressible, which means that every rational map $X \dasharrow X$ is dominant. For example, this is so when $X$ is an anisotropic quadric of dimension $2^n-1$; this explains why anisotropic non-degenerate quadratic forms of dimension $2^n+1$ always have first Witt index equal to one.

Variants of the degree formula were initially considered by Voevodsky \cite{Voe-96} as a step in the proof of the Milnor conjecture (see \cite{Rost-icm}). One can prove the degree formula using cohomology operations \cite{Mer-St-03}, algebraic cobordism \cite[\S4.4]{LM-Al-07}, or by more elementary means \cite{Mer-df-notes}. But, until now, no proof of the formula was valid over a field of characteristic two. In this paper we provide a proof working in arbitrary characteristic, in the spirit of the elementary approach of Rost and Merkurjev.

The degree formula will appear as a special case of a more general result concerning a certain invariant of involutions of varieties. Even if one is only interested in the degree formula, we believe that this more general setting provides flexibility, and makes many arguments and constructions more natural.

The datum of an involution on a variety is equivalent to that of an action of the group $G$ with two elements. We call a $G$-variety ``pure'' if its fixed locus is an effective Cartier divisor. Any $G$-variety can be approximated, in an appropriate sense, by a pure $G$-variety, and many arguments in the paper proceed by reduction to the pure case. We associate to each $G$-variety $Y$ a cycle class $\St{Y}$ in the Chow group of its fixed locus. The proof that this class, modulo two, is functorial with respect to proper morphisms with a degree occupies the main part of the paper. When $X$ is a variety, and $G$ acts by exchange on $X \times X$, the class $\St{X \times X}$ may be identified with the total Segre class $\et{X} \in \CH(X)$ of the tangent cone of $X$. The degree formula amounts to the functoriality of the degree $\etd{X}$ of the component of degree zero of $\et{X}$, when $X$ is complete. But the cycle class $\et{X}$ carries much more information than its degree, and moreover is also defined when $X$ is not complete: we expect $\et{X}$ modulo two to be the total homological Steenrod square of the fundamental class of $X$ (Steenrod squares have been constructed only in characteristic different from two).\\

The paper is organised as follows. We provide in \S\ref{sect:Notations} the notation and conventions used in the paper, and we define in \S\ref{sect:degree} the degree of a morphism, and give some of its elementary properties. In \S\ref{sect:G}, we consider actions of the group $G$ with two elements on varieties. We recall in \S\ref{sect:Gquotients} some facts on $G$-quotients, introduce in \S\ref{sect:fixed} a certain sheaf $\Lg{Y}$ on the quotient $Y/G$ of a $G$-variety $Y$, define pure $G$-varieties in \S\ref{sect:pure}, and discuss how blowing-up closed subschemes affects $G$-actions in \S\ref{sect:antisym}. Next, in \S\ref{sect:res}, we associate to a $G$-equivariant morphism $f\colon Y \to X$ a certain closed subscheme $R_f$ of $Y$, which may be thought of as the bad locus of $f$ with respect to the $G$-actions. Then we define in \S\ref{sect:P} a cycle class $\Std{Y}$ in the Chow group of the fixed locus of a pure $G$-variety $Y$, and prove that it is functorial modulo two with respect to proper morphisms having a degree (\autoref{lemm:main}). This is the main technical step in the paper; it is achieved by gaining some control on the cycles supported on the locus $R_f$. We generalise this functoriality to arbitrary $G$-varieties in \S\ref{sect:P}, for which the cycle class is denoted by $\St{Y}$ instead of $\Std{Y}$, and obtain the main result of the paper (\autoref{th:main_G}). In \S\ref{sect:exch}, we specialise to the $G$-action by exchange on the square of a variety, compute the class $\Sq(X)=\St{X \times X}$ when $X$ is smooth (\autoref{prop:sq_c}), and prove the degree formula (\autoref{cor:rost_df}). Finally we recall classical consequences of the formula in \S\ref{sect:applications}.\\

\paragraph{\textbf{Acknowledgements.}} I would like to acknowledge the crucial influence of the ideas of Markus Rost, and especially those appearing in the introduction of the preprint \cite{Ros-On-08}. In particular his proof that the Segre number is even in characteristic two was the starting point of this work. I also drew inspiration from the construction of Steenrod squares given in \cite[\S2]{Vis-SymRu} and \cite[Chapter~XI]{EKM} (related constructions are performed in \cite[\S5]{Vi-Sym}). I thank the referee for his careful reading of the paper, and his constructive suggestions.

\section{Notation}
\label{sect:Notations}

\subsection{Varieties}
We work over a fixed base field $k$. By a scheme, we will mean a separated scheme of finite type over $k$, and  by a morphism, a $k$-morphism. A \emph{variety} is a quasi-projective scheme over $k$. A scheme is complete if it is proper over $k$. When $X$ and $Y$ are schemes, a rational map $Y \dasharrow X$ is a morphism $U\to X$, for some open dense subscheme $U$ of $Y$. The function field of an integral scheme $X$ will be denoted by $k(X)$. When $Y$ and $X$ are two schemes, we denote by $Y \times X$ their fibre product over $k$.

We say that a morphism $f\colon Y\to X$ is an isomorphism off a closed subset $Z$ of $X$, if the morphism $Y - f^{-1}Z \to X -Z$ obtained by base change is an isomorphism. 

\subsection{Group of cycles} \label{sect:Zo} For a scheme $X$, we let $\Zo(X)$ be the free abelian group generated by the elements $[V]$, where $V$ runs over the integral closed subschemes of $X$. When $f \colon Y \to X$ is a morphism, a group homomorphism $f_* \colon \Zo(Y) \to \Zo(X)$ is defined as follows. Let $W$ be an integral closed subscheme of $Y$, and $V$ be the (scheme theoretic) closure of $f(W)$ in $X$. Then we let
\[
f_*[W] =\left\{ \begin{array}{rl}
		  [k(W):k(V)]\cdot [V] &\mbox{ if $\dim W =\dim V$}, \\
		  0 &\mbox{ otherwise}.
       		\end{array}
	\right.
\]
One checks that this defines a functor from the category of schemes to the category of abelian groups. When $f$ is an immersion, the morphism $f_*$ is injective, and we will view $\Zo(Y)$ as a subgroup of $\Zo(X)$.

The group $\Zo(X)$ is the direct sum over $n \in \mathbb{N}$ of the subgroups $\Zo_n(X)$ generated by the elements $[V]$ such that $\dim V=n$. When $T$ is a closed subscheme of $X$, with irreducible components $T_\alpha$ and multiplicities $m_\alpha$, we define the class $[T]=\sum_\alpha m_\alpha [T_\alpha] \in \Zo(X)$. When $f\colon Y \to X$ is a flat morphism of constant relative dimension, there is a morphism $f^* \colon \Zo(X) \to \Zo(Y)$ such that $f^*[Z] = [f^{-1}Z]$ for any closed subscheme $Z$ of $X$ (see \cite[Lemma~1.7.1]{Ful-In-98}).

The quotient of $\Zo(X)$ by rational equivalence is $\CH(X)$, the Chow group of $X$. We will use its functorialities described in \cite{Ful-In-98} (where the notation $A_*$ is used instead of $\CH$).

\subsection{Flat morphisms of finite rank}
\label{sect:rank}
A morphism $f \colon Y \to X$ is \emph{flat of rank $n$} if it is finite, and $f_*\Oc_Y$ is a locally free $\Oc_X$-module of rank $n$. Then:
\begin{enumerate}[(i)]
\item \label{rank:bc} Any base change of $f$ is flat of rank $n$.

\item \label{rank:flat} The morphism $f$ is faithfully flat, of relative dimension $0$.

\item \label{rank:deg2} The endomorphism $f_* \circ f^*$ of the group $\Zo(X)$ is multiplication by $n$ (see \cite[Example~1.7.4]{Ful-In-98}).
\end{enumerate}

\subsection{Sheaves of ideals}
\label{sect:ideals}
When $Z$ is a closed subscheme of a scheme $X$, we denote by $\Ic_Z$ the corresponding ideal of $\Oc_X$. Let $f\colon Y \to X$ be a morphism. If $\Ic$ is an ideal of $\Oc_X$, we denote by $\Ic \Oc_Y$ the image of the morphism of $\Oc_Y$-modules $f^*\Ic \to \Oc_Y$. Thus when $Z$ is a closed subscheme of $X$, we have $\Ic_Z \Oc_Y = \Ic_{f^{-1}Z}$. When $\Ic,\Jc$ are two ideals of $\Oc_X$, we have $(\Ic \Oc_Y) \cdot (\Jc \Oc_Y)=(\Ic \cdot \Jc)\Oc_Y$.

\section{The degree of a morphism}
\label{sect:degree}

\begin{definition}
\label{def:degree}
Let $f \colon Y \to X$ be a morphism, and $n$ an integer. We say that $f$ \emph{has degree $n$} and write $n=\deg f$, if we have $f_*[Y]=n\cdot[X]$ in $\Zo(X)$. More generally, when $f \colon Y \dasharrow X$ is a rational map, we say that $f$ \emph{has degree $n$} if it induces a morphism of degree $n$ on an open dense subscheme of $Y$.
\end{definition}

\begin{example}
\label{ex:degree}
Let $f\colon Y \dasharrow X$ be a rational map.
\begin{enumerate}[(i)]
\item \label{ex:deg:0} Assume that $f$ is birational, e.g.\ an open dense embedding. Then $f$ has degree one.

\item \label{ex:deg:1} Let $X$ be a scheme with irreducible components $X_\alpha$ and multiplicities $m_\alpha$. Let $Y=\amalg_\alpha X_\alpha'$, where $X_\alpha'$ is the disjoint union of $m_\alpha$ copies of $X_\alpha$, and $f \colon Y \to X$ the natural morphism. Then $f$ has degree one.

\item \label{ex:deg:2} Assume that $Y$ has pure dimension $> \dim X$. Then $f$ has degree zero.

\item \label{ex:deg:3} Assume that $X$ and $Y$ are integral of the same dimension, and that $f$ is dominant. Then $f$ has degree $[k(Y) : k(X)]$.

\item \label{ex:deg:4} Assume that $f=g\circ h$, and that $h$ has degree $n$. Then $f$ has degree $d$ if and only if $d$ is divisible by $n$ and $g$ has degree $d/n$.

\item \label{ex:deg:5} Assume that $f$ is a flat morphism of rank $n$ (see \S\ref{sect:rank}). Then $f$ has degree $n$ (this follows from \S\ref{sect:rank} \eqref{rank:deg2}).
\end{enumerate}
\end{example}

\begin{lemma}
\label{lemm:flat_pullback}
Consider a commutative square
\[ \xymatrix{
Y'\ar[r]^{f'} \ar[d]_y & X' \ar[d]^x \\ 
Y \ar[r]^f & X
}\]
and assume that $f$ has a degree. Then $f'$ has degree $\deg f$ under any of the following assumptions.
\begin{enumerate}[(i)]
\item \label{it:pullback_i} The square is cartesian, and the morphisms $x$ and $y$ are flat of the same constant relative dimension.

\item \label{it:pullback_ii} There is a closed subscheme $Z$ of $X$ such that $x$ induces an isomorphism off $Z$, and $x^{-1}Z$ is nowhere dense in $X'$ and $(f \circ y)^{-1}Z$ is nowhere dense in $Y'$.
\end{enumerate}
\end{lemma}
\begin{proof}
\eqref{it:pullback_i}: This follows immediately from the relation $x^* \circ f_*=f'_* \circ y^*$, as morphisms $\Zo(Y) \to \Zo(X')$ (see \cite[Proposition 49.20 (1)]{EKM}).

\eqref{it:pullback_ii}: Let $U=X -Z$ and $V=Y - f^{-1}Z$. The morphism $V \to U$ has degree $\deg f$ by \eqref{it:pullback_i}. By the assumption \eqref{it:pullback_ii}, we have a commutative diagram
\[ \xymatrix{
V\ar[r] \ar[d] & U \ar[d] \\ 
Y' \ar[r]^{f'} & X'
}\]
where vertical arrows are dense open immersions, and in particular have degree one. Using \autoref{ex:degree}~\eqref{ex:deg:4} twice, we deduce that $f'$ has degree $\deg f$.
\end{proof}

\begin{lemma}
\label{lemm:deg_prod}
For $i=1,2$, let $f_i\colon Y_i \to X_i$ be a morphism with a degree between equidimensional schemes. Then $f_1 \times f_2$ has degree $(\deg f_1) \cdot (\deg f_2)$.
\end{lemma}
\begin{proof}
The morphism $f_1 \times \id_{Y_2}$ (resp.\ $\id_{X_1} \times f_2$) has degree $\deg f_1$ (resp.\ $\deg f_2$) by \autoref{lemm:flat_pullback}~\eqref{it:pullback_i}. By \autoref{ex:degree}~\eqref{ex:deg:4}, it follows that the composite $f_1 \times f_2=(\id_{X_1} \times f_2) \circ (f_1 \times \id_{Y_2})$ has degree $(\deg f_1) \cdot (\deg f_2)$.
\end{proof}

\begin{lemma}
\label{lemm:deg_Cartier}
Consider a cartesian square
\[ \xymatrix{
E\ar[r] \ar[d]_g & Y \ar[d]^f \\ 
D \ar[r] & X
}\]
where horizontal arrows are effective Cartier divisors. Assume that $X$ and $Y$ are equidimensional, and that $f$ is proper. If $f$ has a degree, then $g$ has degree $\deg f$.
\end{lemma}
\begin{proof}
Let $n=\dim Y$, and $Z$ be the scheme theoretic image of $f$. For every $n$-dimensional irreducible component $V$ of $Z$, the cycle $f_*[Y] -[V] \in \Zo(X)$ is effective. Assume that $\dim X >n$. Then if $f$ has a degree, it must be zero, hence $f_*[Y]=0$. Thus there is no $V$ as above, which means that $\dim Z< n$. Since $E=f^{-1}D$ contains no generic point of $Y$, it follows that $D$ contains no generic point of $Z$. Thus $\dim D \cap Z < \dim Z$. Since $E$ has pure dimension $n-1> \dim D \cap Z$, the morphism $E \to D \cap Z$ has degree zero (\autoref{ex:degree}~\eqref{ex:deg:2}), hence the same is true for the morphism $g \colon E \to D \cap Z \to D$, and the statement is proved. Therefore we may assume that $\dim X \leq n$. We have in $\CH_{n-1}(D)$ the relation (see \cite[Definition~2.3, Proposition~2.3 (c), Lemma~1.7.2]{Ful-In-98})
\[
g_* [E] = g_* (E\cdot [Y]) = D \cdot f_*[Y] = D\cdot (\deg f \cdot [X]) = \deg f \cdot [D].
\]
Since $\dim D =\dim X-1\leq n-1$, the morphism $\Zo_{n-1}(D) \to \CH_{n-1}(D)$ is injective, proving the statement.
\end{proof}

\section{Involutions of varieties}
\label{sect:G}
\numberwithin{theorem}{subsection}
\numberwithin{lemma}{subsection}
\numberwithin{proposition}{subsection}
\numberwithin{corollary}{subsection}
\numberwithin{example}{subsection}
\numberwithin{notation}{subsection}
\numberwithin{definition}{subsection}
\numberwithin{remark}{subsection}

A $k$-involution of a variety is called a \emph{$G$-action} (we think of $G$ as the group with two elements). A \emph{$G$-variety} is a variety equipped with a $G$-action, and a \emph{$G$-morphism} is a morphism of varieties compatible with the involutions. A closed or open subscheme of a $G$-variety is a \emph{$G$-subscheme} if it coincides with its inverse image under the involution. When $Y$ and $X$ are two $G$-varieties, then $Y \times X$ is naturally a $G$-variety, via the component-wise $G$-action.

\subsection{\texorpdfstring{$G$}{G}-quotients}\mbox{}
\label{sect:Gquotients}
Let $Y$ be a $G$-variety, $\tau$ its involution, and $y$ a point of $Y$. Since $Y$ is quasi-projective over $k$, we can find an affine open subscheme $U$ of $Y$ containing both $y$ and $\tau(y)$. Then $U \cap \tau^{-1}U$ is an affine open $G$-subscheme of $Y$ containing $y$. Thus:
\begin{proposition}{\cite[V, Proposition 1.8, Corollaire 1.5]{sga1}}
\label{prop:Gquotient}
Let $Y$ be a $G$-variety. The coequaliser of the involution and the identity of $Y$ exists in the category of schemes (separated and of finite type over $k$), and is represented by a finite surjective morphism $\varphi_Y \colon Y \to Y/G$.
\end{proposition}
A $G$-morphism of varieties $f\colon Y\to X$ induces a morphism $f/G \colon Y/G \to X/G$.

\subsection{Fixed locus}
\label{sect:fixed}
\begin{definition}
Let $Y$ be a $G$-variety. The equaliser of the identity and the involution of $Y$ in the category of varieties is a closed embedding $Y\inv \to Y$, called the \emph{fixed locus}. It can be constructed as the fibre product of the graph of the involution $Y \to Y \times Y$ and the diagonal $Y \to Y \times Y$.

If $f\colon Y \to X$ is a $G$-morphism, then $Y\inv \subset f^{-1}(X\inv)$ as closed subschemes of $Y$. We denote by $f\inv \colon Y\inv \to X\inv$ the induced morphism. Note that $f\inv$ is proper as soon as $f$ is so.
\end{definition}
\begin{remark}
\label{rem:product}
Let $Y$ and $X$ be two $G$-variety. Then $(Y\times X)\inv = Y\inv \times X\inv$.
\end{remark}
\begin{example}
\label{ex:product}
Let $X$ be a variety. The variety $X \times X$ has a natural $G$-action, given by the exchange of factors. Then the fixed locus $(X \times X)\inv$ is the diagonal $X$ in $X\times X$. If $f\colon Y \to X$ is a morphism of varieties, then $f\times f \colon Y \times Y \to X \times X$ is a $G$-morphism, and $(f \times f)\inv =f$.
\end{example}

\begin{definition}
We say that a $G$-morphism $f\colon Y \to X$ is \emph{FPR} (fixed-point reflecting) if $f^{-1}(X\inv) = Y\inv$, as closed subschemes of $Y$.
\end{definition}

\begin{remark}
\label{rem:fixed}
The embedding of a $G$-subscheme is FPR.
\end{remark}

\begin{definition}
\label{def:coinv}
Let $Y$ be a $G$-variety, with quotient morphism $\varphi_Y\colon Y \to Y/G$. The $G$-action on $Y$ induces an involution $\sigma_Y$ of the $\Oc_{Y/G}$-module ${\varphi_Y}_*\Oc_Y$. The kernel of the endomorphism $\id -\sigma_Y$ of ${\varphi_Y}_*\Oc_Y$ is $\Oc_{Y/G}$ (see \cite[V, Corollaire~1.2]{sga1}), and we denote by $\Lg{Y}$ its image. Thus $\Lg{Y}$ is a coherent $\Oc_{Y/G}$-submodule of ${\varphi_Y}_*\Oc_Y$, and we have an exact sequence of $\Oc_{Y/G}$-modules
\begin{equation}
\label{seq:l}
0 \to \Oc_{Y/G} \to {\varphi_Y}_*\Oc_Y \to \Lg{Y} \to 0.
\end{equation}

If $f\colon Y \to X$ is a $G$-morphism, so is $Y \to (Y/G) \times_{X/G} X$, and therefore there is a unique morphism $\lambda_f$ fitting into the commutative diagram of $\Oc_{Y/G}$-modules
\begin{equation}
\label{diag:functorial}
\begin{gathered}
\xymatrix{
(f/G)^* {\varphi_X}_* \Oc_X \ar[d] \ar@{->>}[r] & (f/G)^*\Lg{X}\ar[d]_{\lambda_f} \ar[r] & (f/G)^*{\varphi_X}_* \Oc_X \ar[d]\\
{\varphi_Y}_*\Oc_Y\ar[r] & \Lg{Y}\ar@{^{(}->}[r] & {\varphi_Y}_*\Oc_Y
}
\end{gathered}
\end{equation}

\end{definition}
\begin{remark}
\label{rem:Lg_open}
When $Y$ is a $G$-variety, and $u \colon U \to Y$ the open immersion of a $G$-subscheme, then the morphism $\lambda_u \colon (u/G)^*\Lg{Y} \to \Lg{U}$ is an isomorphism. 
\end{remark}

\begin{proposition}
\label{lemm:altdef_fixed}
Let $Y$ be a $G$-variety. We have an exact sequence of $\Oc_Y$-modules
\[
\varphi_Y^*\Lg{Y} \to \Oc_Y \to \Oc_{Y\inv} \to 0,
\]
where the map on the left is adjoint to the inclusion $\Lg{Y} \to {\varphi_Y}_* \Oc_Y$.
\end{proposition}
\begin{proof}
Let us write $\varphi$ for $\varphi_Y$. The morphism $Y\inv \to Y$ is also the equaliser of the identity and the involution of $Y$ in the category of affine $Y/G$-schemes. Therefore the morphism $\varphi_*\Oc_Y \to \varphi_*\Oc_{Y\inv}$ is the coequaliser of the identity and the involution of $\varphi_*\Oc_Y$ in the category of quasi-coherent $\Oc_{Y/G}$-algebras. This means that the $\Oc_{Y/G}$-algebra $\varphi_*\Oc_{Y\inv}$ is the quotient of $\varphi_*\Oc_Y$ by the ideal generated by $\im(\id-\sigma_Y)=\Lg{Y}$, that is by the image of the morphism $\varphi_* \varphi^* \Lg{Y} \to \varphi_* \Oc_Y$. Thus
\[
\varphi_*\Oc_{Y\inv} = \coker(\varphi_* \varphi^* \Lg{Y} \to \varphi_* \Oc_Y) = \varphi_* \coker(\varphi^* \Lg{Y} \to \Oc_Y)
\]
($\varphi$ is affine), and therefore $\Oc_{Y\inv} = \coker(\varphi^* \Lg{Y} \to \Oc_Y)$.
\end{proof}

\subsection{Pure actions}
\label{sect:pure}

\begin{definition}
We say that a $G$-variety $Y$ is \emph{pure} if $Y\inv \to Y$ is an effective Cartier divisor.
\end{definition}

\begin{remark}
\label{rem:pure_open}
An open $G$-subscheme of a pure $G$-variety is a pure $G$-variety (this follows from \autoref{rem:fixed}).
\end{remark}

\begin{lemma}
\label{lemm:lambda_f}
Let $f\colon Y \to X$ be a $G$-morphism. Then the following diagram of $\Oc_Y$-modules commutes (we use the notation of \eqref{diag:functorial}).
\[ \xymatrix{
f^* \varphi_X^*\Lg{X}\ar[r] \ar[d]& f^*\Oc_X \ar[dd]\\
\varphi_Y^*(f/G)^*\Lg{X} \ar[d]_{\varphi_Y^*(\lambda_f)} &\\
\varphi_Y^*\Lg{Y}  \ar[r] & \Oc_Y
}\]
\end{lemma}
\begin{proof}
In view of the equivalence of categories between quasi-coherent $\Oc_Y$-algebras and affine $Y$-schemes, the commutativity of the two diagrams below is equivalent.
\[ \xymatrix{
f^* \varphi_X^*{\varphi_X}_*\Oc_X\ar[r] \ar[d]& f^*\Oc_X \ar[dd]&&Y\times_X X \times_{X/G} X& Y \times_X X \ar[l]\\
\varphi_Y^*(f/G)^*{\varphi_X}_*\Oc_X \ar[d] & && Y\times_{Y/G} (Y/G)\times_{X/G} X \ar[u]&\\
\varphi_Y^*{\varphi_Y}_*\Oc_Y  \ar[r] & \Oc_Y && Y \times_{Y/G} Y  \ar[u] & Y\ar[l] \ar[uu]
}\]
The diagram on the right is obviously commutative. The commutativity of the other diagram yields the result, in view of the rightmost square of \eqref{diag:functorial}.
\end{proof}

\begin{proposition}
\label{prop:def_pure}
Let $Y$ be a pure $G$-variety.
\begin{enumerate}[(i)]
\item \label{prop:def_pure:flat} The morphism $\varphi_Y \colon Y \to Y/G$ is flat of rank two.
\item \label{prop:def_pure:line} The $\Oc_{Y/G}$-module $\Lg{Y}$ is locally free of rank one.
\item \label{prop:def_pure:seq} The sequence of $\Oc_Y$-modules below is exact (see \autoref{lemm:altdef_fixed})
\[
0 \to \varphi_Y^*\Lg{Y} \to \Oc_{Y} \to \Oc_{Y\inv} \to 0.
\]
\end{enumerate}
\end{proposition}
\begin{proof}
Let us write $\varphi$ for $\varphi_Y$ and $\Lg{}$ for $\Lg{Y}$ and $\Ic$ for $\Ic_{Y \inv}$. Consider the symmetric $\Oc_{Y/G}$-algebra $\Sym(\Lg{})$ on the module $\Lg{}$, and the variety $P= \Proj(\Sym(\Lg{}))$. We use the description of the set of $Y/G$-morphisms into $P$ given in \cite[(4.2.3)]{ega-2}. Since the morphism of $\Oc_Y$-modules $a\colon \varphi^*\Lg{} \to \Ic$ of \autoref{lemm:altdef_fixed} is surjective and its target is locally free of rank one, there is an induced morphism $\alpha \colon Y \to P$ over $Y/G$. Let $\tau$ be the involution of $Y$. The morphism $\alpha \circ \tau \colon Y \to P$ corresponds to the surjective morphism of $\Oc_Y$-modules
\[
a'\colon \varphi^*\Lg{} \simeq \tau^*\varphi^* \Lg{} \xrightarrow{\tau^*(a)} \tau^*\Ic.
\]
The $G$-morphism $\tau \colon Y \to Y$ is such that $\tau/G=\id_{Y/G}$ and $\lambda_\tau = -\id_{\Lg{}}$ (because $\sigma_Y$ acts as $-\id$ on $\Lg{} = \im(\id -\sigma_Y) \subset \ker(\id + \sigma_Y)$). Applying \autoref{lemm:lambda_f} to the $G$-morphism $\tau$, we see that the diagram with solid arrows 
\[ \xymatrix{
\tau^*\varphi^*\Lg{}\ar@{->>}[rr]^{\tau^*(a)} \ar[d]_{\simeq} && \tau^*\Ic \ar@{.>}[d]^e\ar@{^{(}->}[r] & \tau^*\Oc_Y \ar[d]^{\simeq} \\ 
\varphi^*\Lg{} \ar@{->>}[rr]^a && \Ic \ar@{^{(}->}[r] & \Oc_Y
}\]
anticommutes. Using the indicated bijectivities, surjectivities and injectivities, we deduce the existence of the isomorphism $e$ which makes the right square commute and the left one anticommute. Thus $(-e) \circ a' = a$. It follows that $a$ and $a'$ define the same morphism $Y \to P$, so that $\alpha=\alpha \circ \tau$. By the universal property of the $G$-quotient, the morphism $P \to Y/G$ admits a section $\gamma \colon Y/G \to P$ such that $\alpha = \gamma \circ \varphi$. The morphism $\gamma$ is represented by a surjective morphism of $\Oc_{Y/G}$-modules $g \colon \Lg{} \to \Fc$ with $\Fc$ locally free of rank one. From the relation $\alpha = \gamma \circ \varphi$, we deduce the existence of an isomorphism $\varphi^*\Fc  \simeq \Ic$ identifying the morphisms $a$ and $\varphi^*(g)$. The composite $\varphi^*\Lg{} \xrightarrow{a} \Ic \to \Oc_Y$ is adjoint to the inclusion $\Lg{} \to \varphi_*\Oc_Y$, and factors through $\varphi^*(g)$. It follows that the inclusion $\Lg{} \to \varphi_*\Oc_Y$ factors through $g$, so that $g$ is injective, hence an isomorphism. Thus the $\Oc_{Y/G}$-module $\Lg{}\simeq\Fc$ is locally free of rank one, proving \eqref{prop:def_pure:line}. The exact sequence \eqref{seq:l} shows that the $\Oc_{Y/G}$-module $\varphi_*\Oc_Y$ is locally free of rank two, being an extension of two locally free modules of rank one. Since $\varphi$ is finite, we obtain \eqref{prop:def_pure:flat}. Finally $a$ is an isomorphism (because $\varphi^*(g)$ is one), which proves \eqref{prop:def_pure:seq}. 
\end{proof}

\begin{proposition}
\label{lemm:Gregular}
Let $f\colon Y \to X$ be a $G$-morphism, and consider the morphism of $\Oc_{Y/G}$-modules $\lambda_f \colon (f/G)^*\Lg{X} \to \Lg{Y}$. Then:
\begin{enumerate}[(i)]
\item \label{lemm:Gregular:1} $\lambda_f$ is surjective if and only if $Y \to (Y/G) \times_{X/G} X$ is a closed embedding.

\item \label{lemm:Gregular:2} $\lambda_f$ is bijective if and only if $Y \to (Y/G) \times_{X/G} X$ is an isomorphism.

\item \label{lemm:Gregular:3} If $\lambda_f$ is surjective then $f$ is FPR (i.e.\ $f^{-1}(X\inv)=Y\inv$).

\item \label{lemm:Gregular:4} If $Y$ is pure and $f$ is FPR, then $\lambda_f$ is surjective.
\end{enumerate}
\end{proposition}
\begin{proof}
We have a commutative diagram of $\Oc_{Y/G}$-modules, with exact rows
\[ \xymatrix{
&(f/G)^* \Oc_{X/G}\ar[r] \ar[d]& (f/G)^*{\varphi_X}_*\Oc_X \ar[d] \ar[r] & (f/G)^*\Lg{X}\ar[d]^{\lambda_f} \ar[r] & 0\\
0 \ar[r] &\Oc_{Y/G}  \ar[r] & {\varphi_Y}_*\Oc_Y \ar[r] & \Lg{Y} \ar[r] & 0
}\]
The left vertical arrow is an isomorphism. Since $\varphi_Y$ and $\varphi_X$ are affine, the middle vertical arrow is surjective, resp.\ bijective, if and only if $Y \to (Y/G) \times_{X/G} X$ is a closed embedding, resp.\ an isomorphism. The equivalences \eqref{lemm:Gregular:1} and \eqref{lemm:Gregular:2} now follow from a diagram chase.

We have by \autoref{lemm:lambda_f} a commutative diagram of $\Oc_Y$-modules
\[ \xymatrix{
f^* \varphi_X^*\Lg{X}\ar[r] \ar[d]_{u}& f^*\Oc_X \ar[d] \ar[r] & f^*\Oc_{X\inv}\ar[d]_w \ar[r] & 0\\
\varphi_Y^*\Lg{Y}  \ar[r] & \Oc_Y \ar[r] & \Oc_{Y\inv} \ar[r] & 0
}\]
where $u$ is the composite of $\varphi_Y^*(\lambda_f)$ with an isomorphism. The morphism $w$ is injective if and only if $Y\inv=f^{-1}(X\inv)$. The surjectivity of $u$ is equivalent to that of $\varphi_Y^*(\lambda_f)$. But $\varphi_Y^*(\lambda_f)$ is surjective if and only if $\lambda_f$ is so, because the morphism $\varphi_Y$ is surjective (the support of $\coker \varphi_Y^*(\lambda_f)=\varphi_Y^*(\coker \lambda_f)$ is the inverse image under $\varphi_Y$ of the support of $\coker \lambda_f$). Now the rows in the diagram above are exact by \autoref{lemm:altdef_fixed}, and the middle arrow is an isomorphism. By a diagram chase, we see that the surjectivity of $u$ implies the injectivity of $w$; this proves \eqref{lemm:Gregular:3}. In case $Y$ is pure, the morphism $\varphi_Y^*\Lg{Y} \to \Oc_Y$ is injective by \autoref{prop:def_pure} \eqref{prop:def_pure:seq}, so that the injectivity of $w$ implies the surjectivity of $u$, proving \eqref{lemm:Gregular:4}. 
\end{proof}

\begin{proposition}
\label{prop:pure}
Let $f \colon Y \to X$ be a $G$-morphism between pure $G$-varieties. Assume that $f$ is FPR. Then the morphism of $\Oc_{Y/G}$-modules $\lambda_f\colon (f/G)^*\Lg{X} \to \Lg{Y}$ is an isomorphism, and the following commutative square is cartesian
\[ \xymatrix{
Y\ar[r]^f \ar[d]_{\varphi_Y} & X \ar[d]^{\varphi_X} \\ 
Y/G \ar[r]^{f/G} & X/G
}\]
\end{proposition}
\begin{proof}
The morphism of $\Oc_{Y/G}$-modules $\lambda_f$ is an isomorphism if and only it is surjective, because both its source and target are locally free of rank one. Therefore the statement follows from \autoref{lemm:Gregular} \eqref{lemm:Gregular:4} and \eqref{lemm:Gregular:2}.
\end{proof}

\begin{corollary}
\label{cor:closed}
Let $Y$ be a $G$-variety, and $Z$ a closed $G$-subscheme of $Y$. Assume that both $Y$ and $Z$ are pure. Then the morphism $Z/G \to Y/G$ is a closed embedding.
\end{corollary}
\begin{proof}
Since $Z\inv =  Y\inv \cap Z$, by \autoref{prop:pure} the morphism $Z \to Y$ is the base change of $Z/G \to Y/G$ along $\varphi_Y$. The morphism $\varphi_Y$ is faithfully flat (\autoref{prop:def_pure} \eqref{prop:def_pure:flat}), and the statement follows by descent \cite[(2.7.1, (xii))]{ega-4-2}.
\end{proof}

\subsection{Actions on blow-ups}
\label{sect:antisym}

\begin{definition}
\label{def:antsym}
We say that a closed subscheme $Z$ of a $G$-variety $Y$ is \emph{admissible} if it is $G$-invariant, and if there is an integer $r \geq 1$ and a coherent ideal $\Jc$ of $\Oc_{Y/G}$ such that $(\Ic_Z)^r = \Jc \Oc_Y$.
\end{definition}

\begin{remark}
\label{rem:anti}
The base change of an admissible closed subscheme along a $G$-morphism is a admissible closed subscheme.
\end{remark}

\begin{lemma}
\label{lemm:anti:fix}
Let $Y$ be a $G$-variety. Then $Y\inv$ is an admissible closed subscheme of $Y$.
\end{lemma}
\begin{proof}
The fixed locus is a $G$-invariant closed subscheme. The image of $\Lg{Y}^{\otimes 2} \to ({\varphi_Y}_*\Oc_Y)^{\otimes 2} \to {\varphi_Y}_*\Oc_Y$ is contained in $\Oc_{Y/G} = \ker (\sigma_Y - \id)$, hence defines an ideal $\Jc$ of $\Oc_{Y/G} $. By \autoref{lemm:altdef_fixed}, we have $\Jc \Oc_Y = (\Ic_{Y \inv})^2$.
\end{proof}

\begin{proposition}
\label{prop:blowup_is_pure}
Let $Y$ be a $G$-variety, and $Z$ an admissible closed subscheme of $Y$. Let $y\colon Y' \to Y$ be the blow-up of $Z$ in $Y$. Then there is a unique $G$-action on $Y'$ such that $y$ is a $G$-morphism. Moreover $y$ is FPR (i.e.\ $y^{-1}(Y\inv) = (Y')\inv$).

In addition, the $G$-variety $Y'$ is pure under any of the following assumptions. 
\begin{enumerate}[(i)]
\item \label{prop:blowup_is_pure:inv} $Z=Y\inv$.
\item \label{prop:blowup_is_pure:pure} $Y$ is pure.
\end{enumerate}
\end{proposition}
\begin{proof}
The closed subscheme $Z$ of $Y$ is $G$-invariant, therefore by the universal property of the blow-up, the $G$-action on $Y$ lifts uniquely to a $G$-action on $Y'$ making $y$ a $G$-morphism.

Let $W$ be a closed subscheme of $Y/G$ whose inverse image $\varphi_Y^{-1}W$ is defined in $Y$ by the ideal $(\Ic_Z)^r$, for some integer $r\geq 1$. Then $y\colon Y' \to Y$ is also the blow-up of $\varphi_Y^{-1}W$ in $Y$. Let $b\colon B \to Y/G$ be the blow-up of $W$ in $Y/G$. By the universal property of the blow-up, we obtain a $G$-morphism $Y' \to B$, where $B$ is endowed with the trivial $G$-action. This morphism factors as $Y' \to Y'/G \xrightarrow{g} B$ with $y/G=b\circ g$. The morphism $Y' \to B \times_{Y/G} Y$ is a closed embedding, and factors through $Y' \to (Y'/G) \times_{Y/G} Y$. Therefore the latter is a closed embedding (recall that all schemes are separated). By \autoref{lemm:Gregular}, it follows that $y$ is FPR. In case \eqref{prop:blowup_is_pure:pure}, it follows from \autoref{lemm:blow_cartier} below that $(Y')\inv=y^{-1}(Y\inv) \to Y'$ is an effective Cartier divisor. The same conclusion holds in case \eqref{prop:blowup_is_pure:inv}, since $(Y')\inv$ is the exceptional divisor in $Y'$.
\end{proof}

\begin{lemma}
\label{lemm:blow_cartier}
The base change of an effective Cartier divisor (resp.\ a nowhere dense closed subscheme) under a blow-up morphism is an effective Cartier divisor (resp.\ a nowhere dense closed subscheme).
\end{lemma}
\begin{proof}
Let $D \to Y$ be an effective Cartier divisor (resp.\ nowhere dense closed subscheme), and $b\colon B\to Y$ the blow-up of a closed subscheme $Z$ in $Y$. Since the exceptional divisor $b^{-1}Z \to B$ is an effective Cartier divisor, any associated (resp.\ generic) point $\eta$ of $B$ is contained in $U=B - b^{-1}Z$. Since $b$ maps $U$ isomorphically onto an open subscheme of $Y$, it follows that $b(\eta)$ is an associated (resp.\ generic) point of $Y$, and therefore cannot be contained in the effective Cartier divisor (resp.\ nowhere dense closed subscheme) $D$. Thus no associated (resp.\ generic) point of $B$ is mapped to $D$, and the statement follows.
\end{proof}

\subsection{Residual scheme}
\label{sect:res}
Let $X$ be a scheme, and $D \to Z \to X$ be closed embeddings. Assume that $D \to X$ is an effective Cartier divisor. The \emph{residual scheme to $D$ in $Z$ on $X$} is the unique closed subscheme $R$ of $X$ such that $\Ic_R \cdot \Ic_D = \Ic_Z$, where we use the notations of \S\ref{sect:ideals} (see \cite[Definition~9.2.1]{Ful-In-98}).

\begin{lemma}
\label{lemm:residual}
Consider a commutative diagram with cartesian squares
\[ \xymatrix{
D'\ar[r] \ar[d] & Z' \ar[d]\ar[r] & X'\ar[d]^f \\ 
D \ar[r] & Z \ar[r] & X
}\]
Assume that the horizontal arrows are closed embeddings, and that the horizontal composites are effective Cartier divisors. Let $R$ be the residual scheme to $D$ in $Z$ on $X$, and $R'$ the residual scheme to $D'$ in $Z'$ on $X'$.

Then $R'=f^{-1}R$ as closed subschemes of $X'$.
\end{lemma}
\begin{proof}
In view of \S\ref{sect:ideals}, we have
\[
\Ic_{Z'}=\Ic_Z \Oc_{X'}=(\Ic_D \cdot \Ic_R)\Oc_{X'} = (\Ic_D \Oc_{X'}) \cdot (\Ic_R\Oc_{X'})=\Ic_{D'} \cdot (\Ic_R\Oc_{X'}).
\]
But $\Ic_{R'}$ is the unique sheaf of ideals of $\Oc_{X'}$ such that $\Ic_{Z'} = \Ic_{D'} \cdot \Ic_{R'}$, hence $\Ic_{R'} = \Ic_R\Oc_{X'}$, proving the statement.
\end{proof}

\begin{lemma}
\label{lemm:univ_res}
Let $D\to Z \to X$ be closed embeddings, with $D\to X$ an effective Cartier divisor. Let $R$ be the residual scheme to $D$ in $Z$ on $X$.
\begin{enumerate}[(i)]
\item \label{rem:R_X} As closed subschemes of $X$, we have $R=X$ if and only if $Z=X$.

\item \label{rem:R_X_e} As closed subschemes of $X$, we have $D=Z$ if and only if $R=\varnothing$.

\item \label{rem:R_Cartier} The morphism $R \to X$ is an effective Cartier divisor if and only if $Z \to X$ is so, in which case $Z=D + R$ as Cartier divisors on $X$.

\item \label{item:isom_off_res} The morphism $D \to Z$ is an isomorphism off $R$.

\item \label{item:univ_res} Let $T$ be a closed subscheme of $Z$, and assume that $D \cap T \to T$ is an effective Cartier divisor. Then $T \subset R$ as closed subschemes of $X$.
\end{enumerate}
\end{lemma}
\begin{proof}
\eqref{rem:R_X}, \eqref{rem:R_X_e} and \eqref{rem:R_Cartier} follow at once from the definition.

\eqref{item:isom_off_res}: We apply \autoref{lemm:residual} with $X'=X-R$, and obtain that the residual scheme to $D'=D-(D\cap R)$ in $Z'=Z-R$ on $X'$ is empty, which means by \eqref{rem:R_X_e} that $D'=Z'$ as closed subschemes of $X'$.

\eqref{item:univ_res}: We apply \autoref{lemm:residual} with $D'=D \cap T$ and $Z'=X'=T$, and note that $R'=T$ by \eqref{rem:R_X}.
\end{proof}

\begin{lemma}
\label{lemm:regular}
Let $Z\to Y \to X$ be closed embeddings. Let $f\colon X' \to X$ be the blow-up of $Z$ in $X$, and $R$ the residual scheme to the exceptional divisor $f^{-1}Z$ in $f^{-1}Y$ on $X'$. The blow-up $Y'$ of $Z$ in $Y$ is naturally a closed subscheme of $R$. If the closed embedding $Z \to Y$ is regular, then $Y'=R$ as closed subschemes of $X'$.
\end{lemma}
\begin{proof}
This can be extracted from \cite{Aluffi-Shadows}; let us nonetheless give a proof. Since the effective Cartier divisor $f^{-1}Z \to X'$ restricts to an effective Cartier divisor $(f^{-1}Z) \cap Y' \to Y'$ (namely, the exceptional divisor of $Y'$), and since $Y' \subset f^{-1}Y$, it follows from \autoref{lemm:univ_res}~\eqref{item:univ_res} that $Y' \subset R$ as closed subschemes of $X'$.

Since residual schemes (by \autoref{lemm:residual}) and blow-ups are compatible with open immersions, in order to prove the equality, we may assume that $X=\Spec A$. Then $Y$, resp.\ $Z$, is defined by the ideal $I$, resp.\ $J$, of $A$. The closed subscheme $f^{-1}Y$ of $X'$ is given by the homogeneous ideal $L=\bigoplus_{n \geq 0} I\cdot J^n$ of the Rees algebra $\bigoplus_{n \geq 0} J^n$. The closed subscheme $R$ of $X'$ is given by the homogeneous ideal $\bigoplus_{n \geq 0} I\cdot J^{n-1}$ (we write $J^{-1}=A$), since its degree $n$ component agrees with that of $L(-1)=\bigoplus_{n \geq 1} I\cdot J^{n-1}$ for $n\geq 1$. We have surjections of graded $A/I$-algebras ($\Sym$ denotes the symmetric algebra)
\[
\Sym_{A/I}(J/I) \to \bigoplus_{n\geq 0} J^n / (I\cdot J^{n-1}) \to \bigoplus_{n \geq 0} (J/I)^n.
\]
The map on the right induces the closed embedding $Y' \to R$. If $Z \to Y$ is a regular closed embedding, then the composite above is an isomorphism \cite[\S5, N$^\text{o}$2, Th\'eor\`eme~1, (i) $\Rightarrow$ (iii)]{Bou-AC-10}, whence the statement.
\end{proof}

\begin{notation}
\label{not:Rf}
Let $Y$ be a pure $G$-variety, and $f\colon Y \to X$ a $G$-morphism. We let $R_f$ be the residual scheme to $Y\inv$ in $f^{-1}(X\inv)$ on $Y$, and $\Rg{f}$ the scheme theoretic image of the composite $R_f \to Y \to Y/G$.
\end{notation}

\begin{lemma}
\label{lemm:reg_off_res}
Let $f\colon Y \to X$ be a $G$-morphism between pure $G$-varieties. Then the morphism  $\lambda_f \colon (f/G)^*\Lg{X} \to\Lg{Y}$ restricts to an isomorphism on $Y/G-\Rg{f}$.
\end{lemma}
\begin{proof}
Applying \autoref{lemm:univ_res}~\eqref{item:isom_off_res} we see that $Y\inv \to f^{-1}(X\inv)$ is an isomorphism off $R_f$. Since $R_f \subset \varphi_Y^{-1}\Rg{f}$ as closed subschemes of $Y$, the statement follows from \autoref{prop:pure} (and \autoref{rem:Lg_open}, \autoref{rem:fixed}, \autoref{rem:pure_open}).
\end{proof}

\begin{lemma}
\label{lemm:res_is_pure}
Let $f\colon Y \to X$ be a $G$-morphism. Assume that $Y\inv \to f^{-1}(X\inv)$ is a regular closed embedding, and let $y\colon Y' \to Y$ be the blow-up of $Y\inv$ in $Y$. Then the $G$-varieties $Y'$ and $R_{f \circ y}$ are pure.
\end{lemma}
\begin{proof}
By \autoref{prop:blowup_is_pure} and \autoref{lemm:anti:fix}, the closed subscheme $(Y')\inv$ of $Y'$ is the exceptional divisor, and in particular $Y'$ is pure. The blow-up $B$ of $Y\inv$ in $f^{-1}(X\inv)$ may be identified with a closed $G$-subscheme of $Y'$, by the universal property of the blow-up. By \autoref{lemm:regular}, we have $R_{f \circ y}=B$ as closed subschemes of $Y'$, and therefore as $G$-varieties. But $B$ is pure by \autoref{prop:blowup_is_pure}~\eqref{prop:blowup_is_pure:inv}.
\end{proof}

\begin{lemma}
\label{lemm:factors}
Let $f\colon Y \to X$ be a $G$-morphism, with $Y$ a pure $G$-variety. Then
\[ \xymatrix{
R_f\ar[r] \ar[d]_{\rho} & Y \ar[d]^{\varphi_Y} \\ 
\Rg{f} \ar[r] & Y/G
}\]
is a cartesian square, and the morphism $R_f \to X\inv$ factors through $\rho$.
\end{lemma}
\begin{proof}
Assume first that $R_f$ is pure. Then the square is cartesian and the morphism $R_f/G \to \Rg{f}$ is an isomorphism by \autoref{cor:closed} and \autoref{prop:pure}. Since the $G$-action on $X\inv$ is trivial, the $G$-morphism $R_f \to X\inv$ factors through the $G$-quotient map $\rho$.

When $R_f$ is possibly not pure, the $G$-morphism $f$ factors as $Y \to Y\times X \to X$. Let $B$ be the blow-up of $(Y\inv) \times (X\inv) = (Y\times X)\inv$ in $Y\times X$, and $g \colon B \to X$ the induced morphism. By the universal property of the blow-up, we may view $Y$ as a closed $G$-subscheme of $B$. We have $Y\cap R_g = R_f$ by \autoref{lemm:residual}. Since the morphism $(Y \times X)\inv \to Y \times (X\inv)$ is an effective Cartier divisor, it is a regular closed embedding, hence $R_g$ is pure by \autoref{lemm:res_is_pure}. By the case treated above, we have a commutative diagram
\[ \xymatrix{
R_f\ar[r] \ar[d] & R_g \ar[d] \ar[r] & X\inv\\ 
\Rg{f} \ar[r] & \Rg{g} \ar[ru] &
}\]
proving the second statement. Let $\Rgg{f} = \Rg{f} \times_{Y/G} Y$. As closed subschemes of $B$, we have $R_f \subset \Rgg{f}$, and $\Rgg{f} \subset \Rg{g} \times_{B/G} B = R_g$ (by the pure case above). Since $\Rgg{f} \subset Y$ and $Y\cap R_g = R_f$, this proves that $R_f=\Rgg{f}$. 
\end{proof}

\section{The invariant of an involution}
\subsection{The cycle class \texorpdfstring{$\Std{Y}$}{P}}
\label{sect:P}
Let $Y$ be a pure $G$-variety. We denote by $L_Y$ the line bundle on $Y/G$ whose $\Oc_{Y/G}$-module of sections is the dual of $\Lg{Y}$. 

By \autoref{prop:def_pure}~\eqref{prop:def_pure:seq} we have $\Oc_Y(Y\inv)=\varphi_Y^*L_Y$, and if $Y$ is equidimensional, then by \cite[Proposition 2.6 (b) (d)]{Ful-In-98}, we have
\begin{equation}
\label{eq:c1L}
c_1(\varphi_Y^*L_Y)[Y] = [Y\inv] \in \CH(Y).
\end{equation}

We denote by $\Hg{Y}$ the restriction of $\varphi_Y^*L_Y$ to the closed subscheme $Y\inv$ of $Y$, and define
\begin{equation}
\label{def:Std}
\Std{Y} =c(-\Hg{Y})[Y\inv]= \sum_{n \geq 0} (-1)^n \cdot c_1(\Hg{Y})^n[Y\inv] \in \CH(Y\inv).
\end{equation}

\begin{lemma}
\label{lemm:deg_Greg}
Let $f\colon Y \to X$ be a proper $G$-morphism between equidimensional, pure $G$-varieties. Assume that $f$ has a degree, and is FPR. Then
\[
(f\inv)_*\Std{Y} = \deg f \cdot \Std{X} \in \CH(X\inv).
\]
\end{lemma}
\begin{proof}
By \autoref{prop:pure}, we have $(f/G)^*L_X = L_Y$ and therefore $(f\inv)^*H_X = H_Y$. The morphism $f\inv$ has degree $\deg f$ by \autoref{lemm:deg_Cartier}. Thus, in $\CH(X\inv)$
\[
(f\inv)_* \circ c(-\Hg{Y})[Y\inv] =  c(-\Hg{X}) \circ (f\inv)_* [Y\inv] = \deg f \cdot c(-\Hg{X})[X\inv].\qedhere
\]
\end{proof}

\begin{lemma}
\label{lemm:main}
Let $f\colon Y\to X$ be a proper $G$-morphism between equidimensional, pure $G$-varieties. Assume that $f$ has a degree and that $f^{-1}(X\inv) \to Y$ is an effective Cartier divisor. Then
\[
(f\inv)_*\Std{Y} = \deg f \cdot \Std{X} \in \CH(X\inv)/2.
\]
\end{lemma}
\begin{proof}
When $d\colon D \to Z$ is an effective Cartier divisor, we will denote by $d^*\colon \CH(Z) \to \CH(D)$ the Gysin map defined in \cite[\S2.6]{Ful-In-98}. For any $G$-variety $T$, we write $\gamma_T$ for the closed embedding $T\inv \to T$. 

Consider the commutative diagram (we use \autoref{not:Rf})
\[ \xymatrix{
R_f\ar[rrd]_h\ar[rr]_i\ar@/^2pc/[rrrr]^r&&f^{-1}(X\inv)\ar[rr]_q \ar[d]_g && Y \ar[d]^f \\ 
&&X\inv \ar[rr]^{\gamma_X} && X
}\]
By \autoref{lemm:univ_res}~\eqref{rem:R_Cartier}, the morphism $r$ is an effective Cartier divisor and we have $R_f + Y\inv = f^{-1}(X\inv)$, as Cartier divisors on $Y$. By \cite[Proposition~2.3~(b)]{Ful-In-98} and definition of the Gysin map, we have, as morphisms $\CH(Y) \to \CH(f^{-1}(X\inv))$,
\[
i_* \circ r^* + j_* \circ \gamma_Y^*= q^*,
\]
where $j\colon Y\inv \to f^{-1}(X\inv)$ is the closed embedding. Composing with $g_*$, we obtain, as morphisms $\CH(Y) \to \CH(X\inv)$,
\begin{equation}
\label{eq:sumcartier}
h_* \circ r^* + (f\inv)_* \circ \gamma_Y^*= g_* \circ q^*.
\end{equation}

By \autoref{lemm:factors} we have a cartesian square
\begin{equation}
\label{squ:R}
\begin{gathered}
\xymatrix{
R_f\ar[r]^r \ar[d]_\rho & Y \ar[d]^{\varphi_Y} \\ 
\Rg{f} \ar[r]^s & Y/G
}
\end{gathered}
\end{equation}
hence by \autoref{prop:def_pure} \eqref{prop:def_pure:flat} and \S\ref{sect:rank} \eqref{rank:bc} the morphism $\rho\colon R_f \to \Rg{f}$ is flat of rank two, and the endomorphism $\rho_*\circ \rho^*$ of $\Zo(\Rg{f})$ is multiplication by two (see \S\ref{sect:rank}~\eqref{rank:deg2}). By \autoref{lemm:factors}, the morphism $h$ factors through $\rho$. It follows that
\begin{equation}
\label{eq:imhrho}
h_* \circ \rho^*\CH(\Rg{f}) \subset 2\CH(X\inv).
\end{equation}

Since $Y$ and $X$ are equidimensional, we have by \cite[Proposition 2.6 (d),(e)]{Ful-In-98}
\begin{equation}
\label{eq:std_xi}
\begin{array}{l}
\Std{Y}=\gamma_Y^* \circ c(-\varphi_Y^*L_Y)[Y] \in \CH(Y\inv),\\
\Std{X}=\gamma_X^* \circ c(-\varphi_X^*L_X)[X] \in \CH(X\inv).
\end{array}
\end{equation}

Using once more \cite[Proposition 2.6 (d)]{Ful-In-98}, we have in $\CH(X\inv)$ 
\begin{align*}
h_* \circ r^*\circ c(-\varphi_Y^*L_Y)[Y] 
&= h_* \circ c(-r^* \varphi_Y^*L_Y) \circ r^*[Y] \\ 
&= h_* \circ c(-r^* \varphi_Y^*L_Y) [R_f] \\
&= h_* \circ c(-\rho^* s^*L_Y) [R_f] \\
&= h_* \circ c(-\rho^* s^*L_Y) \circ \rho^*[\Rg{f}] \\
&= h_* \circ \rho^* \circ c(-s^*L_Y)[\Rg{f}],
\end{align*}
and using \eqref{eq:imhrho}, we obtain
\begin{equation}
\label{eq:qr}
h_* \circ r^*\circ c(-\varphi_Y^*L_Y)[Y] \in 2 \CH(X\inv).
\end{equation}

On the open complement of $\Rg{f}$ in $Y/G$, the line bundles $(f/G)^*L_X$ and $L_Y$ are isomorphic by \autoref{lemm:reg_off_res}. By the localisation sequence, it follows that $(c(-L_Y) - c(-(f/G)^*L_X))[Y/G] \in s_*\CH(\Rg{f})$, and 
\begin{align*}
(c(-\varphi_Y^*L_Y) - c(-f^*\varphi_X^*L_X))[Y]
&= (c(-\varphi_Y^*L_Y) - c(-f^*\varphi_X^*L_X)) \circ \varphi_Y^*[Y/G]\\
&= (c(-\varphi_Y^*L_Y) - c(-\varphi_Y^*(f/G)^*L_X)) \circ \varphi_Y^*[Y/G]\\
&= \varphi_Y^*\circ (c(-L_Y) - c(-(f/G)^*L_X))[Y/G]\\
&\in \varphi_Y^* \circ s_*\CH(\Rg{f}).
\end{align*}
In view of the cartesian square \eqref{squ:R}, we have $\varphi_Y^* \circ s_* = r_*\circ \rho^*$, so that
\begin{equation}
\label{eq:xixi}
(c(-\varphi_Y^*L_Y) - c(-f^*\varphi_X^*L_X))[Y] \in r_*\circ \rho^*\CH(\Rg{f}).
\end{equation}

Since $\Oc_Y(f^{-1}(X\inv)) \simeq f^*\varphi_X^*L_X$, we have, as morphisms $\CH(\Rg{f}) \to \CH(X\inv)$,
\begin{align*}
 g_* \circ q^* \circ r_*\circ \rho^* 
 &= g_* \circ q^* \circ q_* \circ i_* \circ \rho^* \\ 
 &= g_* \circ c_1(q^*f^*\varphi_X^*L_X) \circ i_* \circ \rho^* &&\text{by \cite[Proposition~2.6 (c)]{Ful-In-98}}\\
 &= g_* \circ i_* \circ c_1(i^*q^*f^*\varphi_X^*L_X) \circ \rho^*\\
 &= g_* \circ i_* \circ c_1(\rho^*s^*(f/G)^*L_X) \circ \rho^*\\
 &= g_* \circ i_* \circ \rho^* \circ c_1(s^*(f/G)^*L_X)\\
 &= h_* \circ \rho^* \circ c_1(s^*(f/G)^*L_X),
\end{align*} 
it follows from \eqref{eq:imhrho} that
\begin{equation}
\label{eq:gqrrho}
g_* \circ q^* \circ r_*\circ \rho^*\CH(\Rg{f}) \subset 2 \CH(X\inv).
\end{equation}

Combining \eqref{eq:xixi} and \eqref{eq:gqrrho}, we obtain, in $\CH(X\inv)/2$,
\begin{equation}
\label{eq:qxi}
g_* \circ q^* \circ c(-\varphi_Y^*L_Y)[Y] = g_* \circ q^* \circ c(-f^*\varphi_X^*L_X)[Y].
\end{equation}

Finally, we compute, in $\CH(X\inv)/2$,
\begin{align*}
(f\inv)_*\Std{Y}
&=(f\inv)_* \circ \gamma_Y^* \circ c(-\varphi_Y^*L_Y)[Y] && \text{by \eqref{eq:std_xi}}\\
&=g_* \circ q^*\circ c(-\varphi_Y^*L_Y)[Y] - h_* \circ r^* \circ c(-\varphi_Y^*L_Y)[Y] && \text{by \eqref{eq:sumcartier}}\\ 
 &= g_* \circ q^*\circ c(-\varphi_Y^*L_Y)[Y]&& \text{by \eqref{eq:qr}}\\
 &= g_* \circ q^*\circ c(-f^*\varphi_X^*L_X)[Y]&& \text{by \eqref{eq:qxi}}\\
 &= \gamma_X^* \circ f_* \circ c(-f^*\varphi_X^*L_X)[Y] && \text{by \cite[Prop.~2.3 (c)]{Ful-In-98}}\\
 &= \gamma_X^* \circ c(-\varphi_X^*L_X) \circ f_*[Y]\\
 &= \deg f \cdot \gamma_X^* \circ c(-\varphi_X^*L_X)[X]\\
 &=\deg f \cdot \Std{X}&& \text{by \eqref{eq:std_xi}}.\qedhere
\end{align*} 
\end{proof}

\subsection{The cycle class \texorpdfstring{$\St{Y}$}{S}}
\label{sect:S}
We endow the variety $A=\Au \times \Au$ with the $G$-action by exchange. Then $A\inv$ is the diagonal $\Au$ (\autoref{ex:product}). Let $Y$ be a $G$-variety, and consider the $G$-variety $Y \times A$. We have $(Y \times A)\inv=Y\inv \times \Au$ (\autoref{rem:product}). Let $b_Y \colon P(Y) \to Y \times A$ be the blow-up of $(Y \times A)\inv$ in $Y \times A$. By \autoref{prop:blowup_is_pure} and \autoref{lemm:anti:fix}, the $G$-variety $P(Y)$ is pure and the $G$-morphism $b_Y$ is FPR.

Since the flat pull-back $\CH(Y\inv) \to \CH(Y\inv \times \Au)$ is an isomorphism, we may define a cycle class $\St{Y} \in \CH(Y\inv)$ using the formula, in $\CH(Y\inv \times \Au)$,
\begin{equation}
(b_Y\inv)_* \Std{P(Y)}= \St{Y} \times [\Au].
\end{equation}

\begin{proposition}
\label{prop:segre}
Let $Y$ be a $G$-variety. Assume that $Y\inv \to Y$ is a regular closed embedding with normal bundle $N$. Then ($c$ being the total Chern class)
\[
\St{Y} = c( -N)[Y\inv] \in \CH(Y\inv).
\]
\end{proposition}
\begin{proof}
Let $q \colon Y\inv \times \Au \to Y\inv$ be the projection. Since the blow-up morphism $b_Y$ is FPR, the closed subscheme $P(Y)\inv$ of $P(Y)$ is the exceptional divisor of the blow-up, and the morphism $b_Y\inv\colon P(Y)\inv \to Y\inv \times \Au$ is the projective bundle associated to the vector bundle $q^*N \oplus 1$. The line bundle $\Hg{P(Y)}$ is the normal bundle to the closed embedding of $P(Y)\inv$ into $P(Y)$, and thus is isomorphic to the tautological line bundle $\Oc(-1)$ (see \cite[\S B.6.3]{Ful-In-98}). Therefore by \cite[Proposition~4.1~(a)]{Ful-In-98}, we have in $\CH(Y\inv \times \Au)$
\begin{align*}
\St{Y} \times [\Au]
&=(b_Y\inv)_* \circ c(-\Hg{P(Y)})[P(Y)\inv]\\
&= \sum_{n \geq 0} (-1)^n \cdot (b_Y\inv)_* \circ c_1(\Oc(-1))^n [P(Y)\inv]\\
&= \sum_{n \geq 0} (b_Y\inv)_* \circ c_1(\Oc(1))^n [P(Y)\inv]\\
&= c(-q^*N)[Y\inv \times \Au] \\ 
&= c(-N)[Y\inv] \times [\Au].\qedhere
\end{align*}
\end{proof}
\begin{corollary}
\label{cor:P_S}
Let $Y$ be a pure $G$-variety. Then $\Std{Y}=\St{Y} \in \CH(Y\inv)$.
\end{corollary}

\begin{proposition}
\label{prop:deg_Greg}
Let $f\colon Y \to X$ be a proper $G$-morphism between equidimensional varieties. Assume that $f$ has a degree and is FPR. Then
\[
(f\inv)_* \St{Y} = \deg f \cdot \St{X} \in \CH(X\inv).
\]
\end{proposition}

\begin{theorem}
\label{th:main_G}
Let $f\colon Y \to X$ be a proper $G$-morphism between equidimensional varieties. Assume that $f$ has a degree. Then
\[
(f\inv)_* \St{Y} = \deg f\cdot \St{X} \in  \CH(X\inv)/2.
\]
\end{theorem}

\begin{proof}[Proof of \autoref{th:main_G} (resp.\ \autoref{prop:deg_Greg})]
We construct a commutative diagram of $G$-morphisms
\[ \xymatrix{
Y' \ar[r]^h \ar[rd]_g &P(Y) \ar[r]^{b_Y} & Y \times A \ar[d]^{f \times \id_A} \\ 
&P(X) \ar[r]^{b_X} & X \times A
}\]
as follows. Let $Z=(X \times A)\inv=X\inv \times \Au$. Its inverse image $W=(f\times \id_A\circ b_Y)^{-1}Z$ is an admissible closed subscheme of $P(Y)$ by  \autoref{rem:anti} and \autoref{lemm:anti:fix}. We let $h\colon Y' \to P(Y)$ be the blow-up of $W$ in $P(Y)$. Since the $G$-variety $P(Y)$ is pure, so is $Y'$ by \autoref{prop:blowup_is_pure}~\eqref{prop:blowup_is_pure:pure}. We have $b_X^{-1}Z=P(X)\inv$ since $b_X$ is FPR. The existence of the morphism $g$ such that $g^{-1}(P(X)\inv)=h^{-1}W$, and the fact that $g$ is a $G$-morphism, follow from the universal property of the blow-up.

Since $f$ has a degree, the morphism $f \times \id_A$ has degree $\deg f$ by \autoref{lemm:flat_pullback}~\eqref{it:pullback_i}. The closed subschemes $Z$ of $X \times A$ and $(f\times \id_A)^{-1}Z$ of $Y \times A$ are nowhere dense, being contained in the effective Cartier divisors $X \times \Au$ and $Y \times \Au$. By \autoref{lemm:blow_cartier}, the closed subschemes $b_X^{-1}Z$ of $P(X)$ and $h^{-1}W$ of $Y'$ are nowhere dense. It follows from \autoref{lemm:flat_pullback}~\eqref{it:pullback_ii} that $g$ has degree $\deg f$. By \autoref{prop:blowup_is_pure}, the $G$-morphisms $b_Y$ and $h$, and therefore also $b_Y \circ h$, are FPR. Note that in the situation of \autoref{prop:deg_Greg}, the $G$-morphism $g$ is FPR, and $h$ is an isomorphism. We have in $\CH(X\inv \times \Au)/2$ (resp.\ $\CH(X\inv \times \Au)$),
\begin{align*}
(f\inv)_*\St{Y} \times [\Au]
&= (f\inv \times \id_{\Au})_* \circ (b_Y\inv)_* \Std{P(Y)} && \text{by definition of $\St{Y}$}\\
&=(f\inv \times \id_{\Au})_*\circ (b_Y\inv)_* \circ (h\inv)_*\Std{Y'} && \text{by \autoref{lemm:deg_Greg} } \\
 &= (b_X\inv)_* \circ (g\inv)_*\Std{Y'} \\
 &=\deg f \cdot (b_X\inv)_*\circ \Std{P(X)}&& \text{by \autoref{lemm:main} (resp.\ \ref{lemm:deg_Greg})}\\
 &=\deg f \cdot \St{X} \times [\Au] &&\text{by definition of $\St{X}$.}\qedhere
\end{align*} 
\end{proof}

\section{The degree formula}
\numberwithin{theorem}{section}
\numberwithin{lemma}{section}
\numberwithin{proposition}{section}
\numberwithin{corollary}{section}
\numberwithin{example}{section}
\numberwithin{notation}{section}
\numberwithin{definition}{section}
\numberwithin{remark}{section}

\label{sect:exch}
Let $X$ be a variety. We consider $X \times X$ as a $G$-variety via the exchange of factors. The fixed locus is the diagonal $X$ (see \autoref{ex:product}). We define
\[
\Sq(X) = \St{X \times X}\in \CH(X).
\]

\begin{proposition}
\label{prop:sq_c}
Let $X$ be a smooth variety, with tangent bundle $\Tan_X$. Then
\[
\Sq(X) = c(-\Tan_X)[X] \in \CH(X).
\]
\end{proposition}
\begin{proof}
This follows from \autoref{prop:segre}.
\end{proof}

\begin{theorem}
\label{th:df}
Let $f\colon Y \to X$ be a proper morphism between equidimensional varieties. Assume that $f$ has a degree. Then
\[
f_*\Sq(Y) = \deg f \cdot\Sq(X) \in \CH(X) /2.
\]
\end{theorem}
\begin{proof}
By \autoref{lemm:deg_prod}, the morphism $f \times f$ has degree $(\deg f)^2$, which coincides modulo two with $\deg f$. The varieties $Y\times Y$ and $X \times X$ are equidimensional by \cite[(4.2.4, (iii)),(4.2.6)]{ega-4-2}. We conclude with \autoref{th:main_G} applied to the $G$-morphism $f \times f$.
\end{proof}

\begin{corollary}
\label{cor:sq}
The map $[V] \mapsto \et{V}$, where $V$ is an integral scheme, induces a morphism $\Zo(-) \to \CH(-)/2$ of functors (see \S\ref{sect:Zo}) from the category of varieties and proper morphisms to abelian groups.
\end{corollary}
\begin{remark}
\begin{enumerate}[(i)]
\item To prove that this passes to rational equivalence, i.e.\ that the morphism of \autoref{cor:sq} descends to $\CH(-) \to \CH(-)/2$, would give a construction of Steenrod squares. This has been done in \cite[\S59]{EKM} when the characteristic of $k$ differs from two.

\item One may replace the word ``varieties'' by ``schemes'' in \autoref{cor:sq}, using \cite[Propositions 5.2 and 5.3]{2nd} and \cite[Lemma 1.6]{Kimura-Fractional}.
\end{enumerate}
\end{remark}

The index $n_X$ of a scheme $X$ is the positive generator of the image of $(p_X)_* \colon \Zo(X) \to \Zo(\Spec k)=\Zz$ induced by the morphism $p_X\colon X \to \Spec k$. In other words $n_X$ is the g.c.d.\ of the degrees of closed points on $X$. When $X$ is a complete variety, the morphism $(p_X)_*$ above descends to a morphism $\deg \colon \CH(X) \to \Zz$, and we consider the integer
\[
\etd{X} = \deg \Sq(X) \in \Zz.
\]
When $X$ is a smooth, connected, complete variety of dimension $d$, we have by \autoref{prop:sq_c} ($c_d$ denotes the $d$-th Chern class, with values in $\CH_0(X)$)
\begin{equation}
\label{eq:etd}
\etd{X} = \deg c_d(-\Tan_X).
\end{equation}

Taking the degree in \autoref{th:df}, we obtain:
\begin{corollary}
\label{cor:df}
Let $f \colon Y \to X$ be a morphism between complete equidimensional varieties. Assume that $f$ has a degree. Then $n_X \mid n_Y$ and we have in $\Zz/2$
\[
\frac{n_Y}{n_X} \cdot \frac{\etd{Y}}{n_Y} =  \deg f \cdot \frac{\etd{X}}{n_X} \mod 2.
\]
\end{corollary}

The next corollary was proved (for $X$ is smooth) in arbitrary characteristic in \cite{Mer-Ori,reduced}, and in characteristic two in \cite{Ros-On-08}.
\begin{corollary}
\label{cor:even}
Let $X$ be a complete integral variety of positive dimension. Then the integer $s_X$ is even.
\end{corollary}
\begin{proof}
We apply \autoref{cor:df} to the morphism $X \to \Spec k$ of degree zero.
\end{proof}

\begin{corollary}[Rost's degree formula]
\label{cor:rost_df}
Let $f \colon Y \dasharrow X$ be a rational map between smooth, connected, complete varieties of dimension $d$. Then $n_X \mid n_Y$, and we have in $\Zz/2$
\[
\frac{n_Y}{n_X} \cdot \frac{\deg c_d(-\Tan_Y)}{n_Y} = \deg f \cdot \frac{\deg c_d(-\Tan_X)}{n_X} \mod 2.
\]
\end{corollary}
\begin{proof}
The existence of a rational map $Y \dasharrow X$ implies that $n_X \mid n_Y$ (see e.g.\ \cite[Lemma~3.3]{euler} or \cite[Proposition~4.5]{invariants}). Note that $f$ automatically has a degree, since we are in one of the situations \eqref{ex:deg:2} or \eqref{ex:deg:3} of \S\ref{def:degree}. Let $Z$ be the scheme theoretic closure of the graph of $f$ in $Y \times X$. We have proper morphisms $h\colon Z \to Y$ and $g\colon Z \to X$. The statement follows by applying \autoref{cor:df} to the morphisms $h$ (of degree one) and $g$ (of degree $\deg f$), and using \eqref{eq:etd}.
\end{proof}

\section{Applications to incompressibility}
\label{sect:applications}
In this section, we describe classical consequences of the degree formula, which are now also valid in characteristic two. With the exception of \autoref{ex:involution_var}, the arguments are well-known and can be found for instance in \cite{Mer-df-notes,Mer-St-03,Kar-can}.

\begin{definition}[{\cite[p.150]{Kar-can}}]
\label{def:strong_incomp}
Let $p$ be a prime number. A complete integral variety $X$ is called \emph{strongly $p$-incompressible}  if for every complete integral variety $Y$ with $v_p(n_Y) \geq v_p(n_X)$ ($v_p$ is the $p$-adic valuation) and $\dim Y \leq \dim X$, and such that the index $n_{Y_{k(X)}}$ of the $k(X)$-variety $Y_{k(X)}$ is prime to $p$, we have $\dim X = \dim Y$ and the index $n_{X_{k(Y)}}$ of the $k(Y)$-variety $X_{k(Y)}$ is prime to $p$.
\end{definition}
A strongly $p$-incompressible variety (for some $p$) is in particular \emph{incompressible}: every rational map $X \dasharrow X$ is dominant.

\begin{proposition}
Let $X$ be a complete integral variety such that the integer $\etd{X}/n_X$ is odd. Then $X$ is strongly $2$-incompressible.
\end{proposition}
\begin{proof}
Let $Y$ be a complete integral variety with $v_2(n_Y) \geq v_2(n_X)$ such that $n_{Y_{k(X)}}$ is odd and $\dim Y \leq \dim X$. The $k(X)$-variety $Y_{k(X)}$ possesses a closed point of odd degree, let $Z \subset Y \times X$ be its closure. Applying \autoref{cor:df} to the morphism $Z \to X$ (of odd degree), we see that the integers $s_Z/n_Z$ and $n_Z/n_X$ are both odd. Thus $v_2(n_X) = v_2(n_Z)$. Since $\dim Y \leq \dim X = \dim Z$, the morphism $f\colon Z \to Y$ has a degree (\autoref{ex:degree}, \eqref{ex:deg:2}, \eqref{ex:deg:3}). The integer $n_Z/n_Y$ is odd, because of the assumption $v_2(n_Y) \geq v_2(n_X) = v_2(n_Z)$. Applying \autoref{cor:df} to the morphism $f$, we see that $\deg f$  must be odd. In particular $\dim Y = \dim Z = \dim X$, and the base change of $Z \to Y \times X$ along $X_{k(Y)} \to Y \times X$ is a closed point of odd degree of the $k(Y)$-variety $X_{k(Y)}$.
\end{proof}

\begin{example}[Severi-Brauer varieties]
\label{ex:SB}
Let $X$ be the Severi-Brauer variety of a central division $k$-algebra of degree $2^n$. Then $n_X=2^n$ (see \cite[Example~6.1]{Mer-St-03}). Since $X$ becomes isomorphic to the projective space $\mathbb{P}^{2^n-1}$ after some extension of the base field, we have (see \cite[Proposition~7.1]{Mer-df-notes})
\[
\etd{X}=\deg c(-\Tan_X) = \deg c_{2^n-1}(-\Tan_{\mathbb{P}^{2^n-1}}) = \binom{-2^n}{2^n-1} = -\binom{2^{n+1}-2}{2^n-1}.
\]
The dyadic valuation of this integer is $n$. Thus $X$ is strongly $2$-incompressible.

According to \cite[Example~2.3]{Kar-can}, the Severi-Brauer variety of a $p$-primary central division $k$-algebra is strongly $p$-incompressible when the characteristic of $k$ differs from $p$, but it is unknown whether this holds in characteristic $p$. The special case of algebras of prime degree $p$ was later settled in \cite[Corollary~10.2]{firstst}. The general case is thus now settled when $p=2$.
\end{example}

\begin{example}[Quadrics]
\label{ex:quadrics}
Let $\varphi$ be a non-degenerate anisotropic quadratic form of dimension $2^n+1$ with $n\geq 1$. The corresponding projective quadric $X$ is a smooth variety of dimension $2^n-1$. We have $n_X=2$ by Springer's Theorem. By \cite[Lemma~78.1]{EKM} (which is valid in any characteristic), we have
\[
\Sq(X)=c(-\Tan_X) = (1+h)^{-2^n-1} \mod 2 \CH(X),
\]
where $h\in \CH(X)$ is the hyperplane class, a nilpotent element of the ring $\CH(X)$. Using the relation $\deg (h^{2^n-1})=2$, we compute in $\Zz/4$
\[
\etd{X} = 2\binom{-2^n-1}{2^n-1} = -2\binom{2^{n+1}-1}{2^n-1} = 2 \mod 4.
\]
Thus $X$ is strongly $2$-incompressible. In particular $X$ is incompressible, or equivalently by \cite[Theorem~90.2]{EKM} the first Witt index of $\varphi$ is one. More precisely, as explained in \cite[\S5]{Mer-df-notes} and \cite[\S7.5]{Mer-St-03}, one may deduce two theorems on the splitting behaviour of non-degenerate quadratic forms, originally due (in characteristic not two) to Hoffmann \cite{Hoffmann-separation} and Izhboldin \cite{Izhboldin-motivic2}. In characteristic two, they have been proved by a different method in \cite{Hoffmann-Laghribi}. The theorem of Hoffmann implies the following bound on the first Witt index (see \cite[Example~79.7]{EKM}, \cite[Lemma~4.1]{Hoffmann-Laghribi}): any non-degenerate anisotropic quadratic form of dimension $d\geq 2^n +1$ has first Witt index $\leq d-2^n$.
\end{example}

\begin{example}[Involutions varieties]
\label{ex:involution_var}
Let $D$ be a non-trivial finite dimensional central division $k$-algebra equipped with a quadratic pair \cite[\S5.B]{KMRT}, and $X$ the variety of isotropic right ideals of reduced dimension one in $D$ (the involution variety). The variety $X$ is a closed subscheme of the Severi-Brauer variety of $D$. The presence of an involution of the first kind on the algebra $D$ forces its degree to be of the form $2^n$ for some $n \geq 1$ \cite[Corollary~2.8]{KMRT}, and thus $2^n \mid n_X$ by \autoref{ex:SB}. Now $X$ becomes isomorphic to a (split) smooth projective quadric $Q$ of dimension $2^n-2$ after some extension of the base field. Denoting by $h \in \CH(Q)$ the hyperplane class, we see that 
\[
\etd{X} = \etd{Q}= \deg c(-\Tan_Q) = \deg \frac{1+2h}{(1+h)^{2^n}}
\]
(see the proof of \cite[Lemma~78.1]{EKM}). Since $\deg (h^{2^n-2}) =2$, this gives
\[
\etd{X} = 2\binom{-2^n}{2^n-2} + 4 \binom{-2^n}{2^n-3} = 2 \binom{2^{n+1}-3}{2^n-2} - 4\binom{2^{n+1}-4}{2^n-3}.
\]
The dyadic valuation of this integer is $n$. Thus $X$ is strongly $2$-incompressible (it was known to be $2$-incompressible \cite[Proposition~4.5.2]{Zhy-phd}). Note that, in contrast to both the above examples, the integer $\dim X +1$ is not a power of two (unless $D$ is a quaternion algebra, in which case $\dim X=0$).
\end{example}

\end{document}